\documentclass[a4paper,11pt]{article}

\usepackage[T1]{fontenc}

\usepackage{natbib}
\usepackage{graphicx}

\usepackage{amsmath,amssymb,amsthm}
\usepackage{url}
\usepackage[utf8]{inputenc}

\newtheorem{theorem}{Theorem}

\theoremstyle{definition}
\newtheorem{definition}{Definition}

\begin{document}
%\begin{frontmatter}

\title{Kumaraswamy and beta distribution are related by the logistic
  map} \author{B.~Trancón y Widemann\\[1ex] \normalsize
  {\ttfamily Baltasar.Trancon@uni-bayreuth.de}\\
  \normalsize Ecological Modelling, University of Bayreuth, Germany}

\maketitle

\begin{abstract}
  The Kumaraswamy distribution has been proposed as an alternative to
  the beta distribution with more benign algebraic properties. They
  have the same two parameters, the same support and qualitatively
  similar shape for any parameter values. There is a generic
  relationship between the distributions established by a simple
  transformation between arbitrary Kumaraswamy-distributed random
  variables and certain beta-distributed random variables. Here, a
  different relationship is established by means of the logistic map,
  the paradigmatic example of a discrete non-linear dynamical system.
\end{abstract}


%\end{frontmatter}
% \maketitle

%\linenumbers

\section{Introduction}

\subsection{The Logistic Map}

The logistic map is a parametric discrete-time non-linear dynamical
system. It is widely studied because of its complex transition from
orderly to chaotic behavior in spite of an extremely simple defining
equation \citep{may1976}.

\begin{definition}[Logistic map]
  For a real parameter $r > 0$, the logistic map is the function
  \begin{equation}
    \label{eq:logistic}
    f_r(x) = r x (1 - x)
  \end{equation}
  restricted to the closed real interval $[0, 1]$. This is a total
  function for $r \leq 4$.
\end{definition}

The objects of interest are the \emph{trajectories} that arise from
the iteration of $f_r$, i.e., sequences with the recurrence relation
$x_{n+1} = f_r(x_n)$. With increasing $r$, they show all kinds of
behavior from convergence in a singular attractor through bifurcating
periodic solutions to deterministic chaos. The case $r = 4$ is of
special interest: It is known to be chaotic and ergodic on the
interval $[0,1]$. The density of the attractor is given by the beta
distribution with parameters $\alpha = \beta = 1/2$ and hence, by
ergodicity, the long-term distribution of states under almost all
initial conditions (with the exception of unstable periodic solutions)
converges to that distribution \citep{Jakobson1981}.

\subsection{The Kumaraswamy Distribution}

\Citet{Kumaraswamy198079} introduced his now eponymous distribution,
originally called \emph{double-bounded} distribution, as an
alternative to the beta distribution. They have the same real
parameters $\alpha, \beta > 0$, the same support and similar shapes,
but the Kumaraswamy distribution function, unlike the beta
distribution function, has a closed algebraic form. Is has been found
both more accurately fitting hydrological data in simulations
\citep{Kumaraswamy198079} and computationally more tractable
\citep{Jones200970}.

The similarity between the two classes can be formalized.  It is known
and easy to see (from equations \eqref{eq:beta} and
\eqref{eq:kumaraswamy} below) that, if random variable $X$ is
Kumaraswamy-distributed with parameters $\alpha$ and $\beta$, then
$X^\alpha$ is beta-distributed with $\alpha = 1$ and the same
$\beta$. Here, a new formal relationship between certain beta and
Kumaraswamy distributions is established.

\section{Computing State Distributions of the Logistic Map}

Given a random variable $X$ with known distribution on $[0, 1]$, the
propagated distribution of $f_r(X)$ can be computed. In the following,
the term \emph{continuous probability distribution} (cdf) refers to a
cumulative distribution function with support $[0, 1]$, that is, a
continuous weakly monotonic function $F : [0, 1] \to [0, 1]$ with
$F(0) = 0$ and $F(1) = 1$. A random variable $X$ is distributed
according to $F$, written $X \sim F$, if and only if $P(X \leq y) =
P(X < y) = F(y)$.

\begin{definition}[Propagation]
  Let $F$ be a cdf. Then $\widetilde f_r(F)$ is another cdf, namely
  \begin{subequations}
  \begin{equation}
    \label{eq:liftcdf}
    \widetilde f_r(F)(y) = F\bigl(\tfrac 12 - q_r(y)\bigr) + 1 -
    F\bigl(\tfrac 12 + q_r(y)\bigr) \,,
  \end{equation}
  where
  \begin{equation}
    \label{eq:liftroot}
    q_r(y) =
    \begin{cases}
      \sqrt{\tfrac 14 - \tfrac yr} & \text{if } y \leq \tfrac r4
      \\
      0 & \text{if } y > \tfrac r4 \,.
    \end{cases}
  \end{equation}
\end{subequations}
\end{definition}

\begin{theorem}
  \label{theorem:propagation}
  If $X \sim F$ then $f_r(X) \sim \widetilde f_r(F)$.
\end{theorem}

\begin{proof}
  \begin{align*}
    P(f_r(X) \leq y\bigr) &= P\bigl(rx(1-x) \leq y)
    \\
    &= P\bigl(\bigl(x - \tfrac 12\bigr)^2 \geq \tfrac 14 - \tfrac yr\bigr)
    \\
    &= P\bigl(\bigl\lvert x - \tfrac 12\bigr\rvert \geq q_r(y)\bigr)
    \\
    &= P\bigl(x \leq \tfrac 12 - q_r(y) \lor x \geq \tfrac 12 + q_r(y)\bigr)
    \\
    &= P\bigl(x \leq \tfrac 12 - q_r(y)\bigr) + P\bigl(x \geq \tfrac 12 +
    q_r(y)\bigr)
    \\
    &= P\bigl(x \leq \tfrac 12 - q_r(y)\bigr) + 1 - P\bigl(x \leq \tfrac 12 +
    q_r(y)\bigr)
    \\
    &= F\bigl(\tfrac 12 - q_r(y)\bigr) + 1 - F\bigl(\tfrac 12 + q_r(y)\bigr)
    \\
    &= \widetilde f_r(F)(y)
  \end{align*}
\end{proof}

The beta distribution with real parameters $\alpha, \beta > 0$ has
the cdf
\begin{equation}
  \label{eq:beta}
  B(\alpha, \beta)(y) = \int_0^y t^{\alpha-1} (1-t)^{\beta-1} dt \Bigm/ \int_0^1 t^{\alpha-1} (1-t)^{\beta-1} dt \,.
\end{equation}

For $\alpha = \beta = 1/2$ the beta distribution is the arcsine
distribution.
\begin{equation}
  \label{eq:arcsin}
  B(1/2, 1/2)(y) = A(y) = 2/\pi \arcsin \sqrt{y}
\end{equation}

\begin{theorem}
  \label{theorem:arcsin}
  The arcsine distribution is a fixed point of $\widetilde
  f_4$.
\end{theorem}

\begin{proof}
  Using the symmetry of \eqref{eq:arcsin} and the trigonometric
  identity
  \begin{equation}
    \label{eq:arclemma}
    \arcsin \sqrt{y} = 2 \arcsin \sqrt{\frac{1 - \sqrt{1 - y}}2} \,,
  \end{equation}
  one finds that
  \begin{align*}
    \widetilde f_4(A)(y) &= 2A\bigl(\tfrac 12 - q_4(y)\bigr)
    %\\&
    = 2A\left(\frac{1 - \sqrt{1 - y}}2\right)
    %\\&
    = A(y) \,.
  \end{align*}
\end{proof}

The uniform distribution has the cdf $U(y) = y$.  The Kumaraswamy
distribution with real parameters $\alpha, \beta > 0$ has the cdf
\begin{equation}
  \label{eq:kumaraswamy}
  K(\alpha, \beta)(y) = 1 - (1 - y^\alpha)^\beta \,.
\end{equation}

\begin{theorem}
  \label{theorem:kumaraswamy}
  Let $X$ be distributed uniformly. Then $f_4{}^2(X) =
  f_4\bigl(f_4(X)\bigr)$ is Kumaraswamy-distributed with $\alpha =
  \beta = 1/2$.
  \begin{equation}
    \label{eq:main}
    \widetilde f_4{}^2(U) = K\bigl(\tfrac 12, \tfrac 12\bigr)
  \end{equation}
\end{theorem}

\begin{proof}
  For $r=4$, \eqref{eq:liftroot} simplifies to
  \begin{math}
    q_4(y) = \sqrt{1-y}/2
  \end{math}
  and hence \eqref{eq:liftcdf} to
  \begin{equation*}
    \widetilde f_4(F)(y) = F\left(\frac{1-\sqrt{1-y}}2\right) + 1 -
    F\left(\frac{1+\sqrt{1-y}}2\right) \,.
  \end{equation*}
  Then after one step one has
  \begin{align*}
    \widetilde f_4(U)(y) &= \frac{1-\sqrt{1-y}}2 + 1 - \frac{1+\sqrt{1-y}}2
    %\\&
    = 1- \sqrt{1-y} \,,
  \end{align*}
  such that $\widetilde f_4(U) = K(1, 1/2)$, and after two steps one
  has
  \begin{align*}
    \widetilde f_4{}^2(U)(y) &= \widetilde f_4\bigl(\widetilde
    f_4(U)\bigr)(y)
    \\
    &= \widetilde f_4(U)\left(\frac{1-\sqrt{1-y}}2\right) + 1 - \widetilde
    f_4(U)\left(\frac{1+\sqrt{1-y}}2\right)
    \\
    &= 1 - \underbrace{\left(\sqrt{\frac{1+\sqrt{1-y}}2} -
        \sqrt{\frac{1-\sqrt{1-y}}2}\right)}_s \,.
  \end{align*}
  The term $s$ looks rather convoluted, but its square simplifies to
  \begin{math}
    s^2 = 1 - \sqrt{y}
  \end{math},
  so with $s > 0$ conclude that
  \begin{math}
    s = \sqrt{1 - \sqrt{y}}
  \end{math}.
  Hence from \eqref{eq:kumaraswamy} follows \eqref{eq:main}:
  \begin{equation*}
    \widetilde f_4{}^2(U)(y) = 1 - \sqrt{1 - \sqrt{y}} = 1 - (1 - y^{1/2})^{1/2}
  \end{equation*}
  and thus $\widetilde f_4{}^2(U) = K(1/2, 1/2)$.
\end{proof}

\section{Conclusion}

We have examined the probability distribution of successive states of
the chaotic dynamical system described by the logistic map with
parameter value $r=4$. It is known that this system is ergodic and
that its long-term behavior is described by the beta$(1/2,1/2)$
distribution. Hence the distribution of states will converge almost
surely. In order to investigate individual elements of that convergent
sequence, we have given a closed formula for propagating probability
distributions over the iteration of $f_r$
(Theorem~\ref{theorem:propagation}), and validated this approach by
demonstrating that the beta$(1/2, 1/2)$ distribution is indeed a fixed
point solution (Theorem~\ref{theorem:arcsin}).  We have proceeded to
demonstrate that, starting from the uniform distribution (the
maximum-entropy distribution on the state space), after two steps the
Kumaraswamy$(1/2,1/2)$ distribution is reached
(Theorem~\ref{theorem:kumaraswamy}).  Figure~\ref{fig:diagram} shows
the cdfs of the initial state and the first four iterates plotted over
the uniform, Kumaraswamy and beta distributions.

The above findings qualify as purely theoretical results in the sense
that they provide original and provable insights about the
relationship of mathematical objects of theoretical interest. No
suggestion can be given yet as to the application to real-world
statistical problems. Apart from the investigation of potential
practical implications, the presentation of higher iterates
$\widetilde f_4{}^3(U), \dots$ is an open problem. The initial
sequence of distributions $\widetilde f_4{}^n(U)$ looks pretty
regular:
\begin{equation*}
  \underbrace{U = B(1, 1) = K(1, 1)}_{n=0} ~\longmapsto~ \underbrace{K\bigl(1, \tfrac 12\bigr)}_{n=1} ~\longmapsto~ \underbrace{K\bigl(\tfrac 12, \tfrac 12\bigr)}_{n=2} \,,
\end{equation*}
but there is no obvious continuation of the pattern for $n \geq 3$.

\section*{Acknowledgments}

Thanks to Holger Lange, Skog og Landskap, Ås, Norway, for advice on
dynamical systems.

\bibliographystyle{plainnat}
\bibliography{kumaraswamy}

\begin{figure}
  \centering
  \includegraphics[width=\textwidth]{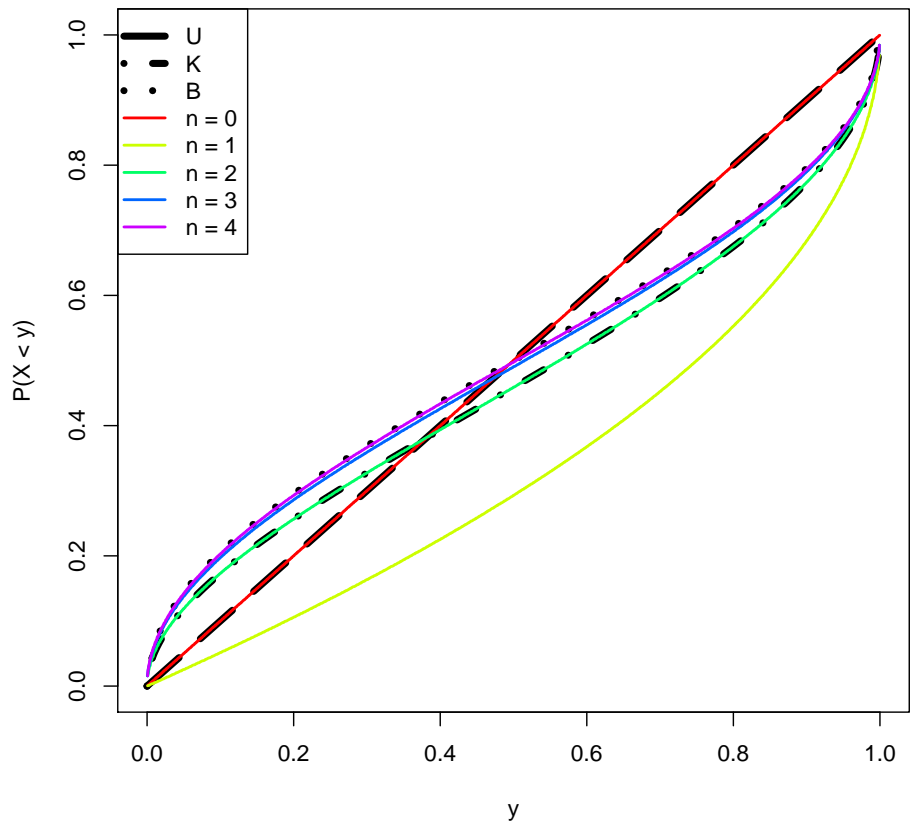}
  \caption{State cdfs $D_n = \widetilde f_4{}^n(U)$ of an initially
    uniformly distributed system running the logistic map with $r=4$,
    after $n = 0, \dots, 4$ steps; plotted over the uniform ($U$),
    Kumaraswamy ($K$) and beta ($B$) distributions, with parameters
    $\alpha = \beta = 1/2$. $D_0$ coincides with $U$. $D_2$ coincides
    with $K$. For $n > 3$, $D_n$ converges towards $B$.}
  \label{fig:diagram}
\end{figure}

\end{document}